\documentclass[12pt,a4paper]{article}

\usepackage{amsmath,amssymb,amsthm,url}
\usepackage{fullpage}
\usepackage{enumerate}
\usepackage{url}
\usepackage{times} 
\usepackage{graphicx}

\usepackage[all]{xy}
\usepackage{xypic}

\theoremstyle{plain} \newtheorem{theorem}{Theorem}[section]
\theoremstyle{plain} \newtheorem{proposition}[theorem]{Proposition}
\theoremstyle{plain} \newtheorem{lemma}[theorem]{Lemma}
\theoremstyle{plain} \newtheorem{corollary}[theorem]{Corollary}
\theoremstyle{plain} 
\theoremstyle{definition} 

\theoremstyle{definition} \newtheorem{definition}[theorem]{Definition}

\newcommand{\mN}{\mathbb{N}}

\newcommand{\mC}{\mathcal{C}}
\newcommand{\mR}{\mathcal{R}}
\newcommand{\mK}{\mathcal{K}}
\newcommand{\mP}{\mathcal{P}}
\newcommand{\mH}{\mathcal{H}}

\newcommand{\mM}{\mathcal{M}}
\newcommand{\id}{\operatorname{id}}
\begin{document}

\title{Construction of the discrete Morse complex in the non-compact case}

\author{Micha\l \ Kukie\l a}
\date{}
\maketitle

\begin{abstract}
We prove an infinite analogue of the main theorem of discrete Morse theory formulated in terms of discrete Morse matchings. Our theorem holds under the assumption that the given Morse matching induces finitely many equivalence classes of infinite directed simple paths. A homological version of the theorem is also given. 
\end{abstract}

\begin{section}{Introduction}
Discrete Morse theory, developed by Robin Forman (see \cite{forman}, or \cite{forman-easy} for a non-technical introduction), is a tool for investigating homotopy type and homology groups of finite CW-complexes.  Forman calls a function $f$ from the set of cells of a regular, finite CW-complex $X$ to the set of real numbers $\mathbb{R}$ a discrete Morse function, if for each cell $\sigma$ the sets \[u_f(\sigma)=\{\tau>\sigma:\dim(\tau)=\dim(\sigma)+1, f(\tau)\leq f(\sigma)\}\] and \[d_f(\sigma)=\{\mu<\sigma:\dim(\mu)=\dim(\sigma)-1, f(\mu)\geq f(\sigma)\}\]have cardinality at most one. A cell $\sigma$ is called critical with respect to $f$ if both $u_f(\sigma)$ and $d_f(\sigma)$ are empty. It turns out, and this result is called the main theorem of discrete Morse theory, that the CW-complex $X$ is homotopy equivalent to a CW-complex that has exactly one $n$-dimensional cell for each critical cell in $X$ of dimension $n$. Moreover, for $a\in\mathbb{R}$ one can define the level subcomplex \[X(a)=\bigcup_{f(\sigma)\leq a}\bigcup_{\tau\leq \sigma}\tau.\] If $f^{-1}([a,b])$ does not contain any critical cells for some $a<b\in\mathbb{R}$, then the level subcomplex $X(b)$ collapses (in the sense of simple homotopy theory \cite{cohen}) to $X(a)$.

The theory has found many applications, both in topological combinatorics and in applied mathematics. It has also been extended in different directions and presented from different viewpoints. For some extensions of the theory developed by Forman himself, consider \cite{forman-novikov,forman-witten}. An approach to equivariant version of the theory was taken by Ragnar Freij \cite{freij}. Gabriel E.~Minian \cite{minian} extends the theory to a class of posets strictly larger than the class of face posets of regular CW-complexes. Some algebraic versions of the theory, such as presented in \cite{freij,jollenbeck}, are also available.

However, most of these extensions deal with finite CW-complexes.  The problem of extending the theory to infinite complexes was stated by Forman in \cite{forman-easy} and tackled with by numerous authors in \cite{ayala} and earlier work cited there. The present paper deals with the same problem, with a slightly different approach. (It is perhaps worth noting that the discrete Morse theory was also applied to infinite complexes in \cite{mathai}, though in a vastly different spirit than in the present paper; we shall not discuss these investigations.) 

To compare our results with those of \cite{ayala} we need to utilise an observation attributed to Manoj K. Chari \cite{chari}, that a discrete Morse function $f$ on a finite CW-complex $X$ induces a matching $M$ on the Hasse diagram $\mH(X)$ of the face poset of $X$. The matching is constructed using the observation that for a cell $\sigma$ of $X$ only one of the sets $u_f(\sigma), d_f(\sigma)$ may be nonempty; its only element is matched with $\sigma$. This matching is acyclic, which means that the digraph $\mH_M(X)$ obtained from $\mH(X)$ by inverting the arrows which belong to the matching does not contain a directed cycle. Those acyclic matchings are in fact a different description of discrete vector fields, studied by Forman \cite{forman}, and they carry all the information on critical cells and the sets $d_f(\sigma), u_f(\sigma)$.

In the case the CW-complex $X$ is finite, given an acyclic matching $M$ on $X$ one may construct a discrete Morse function on $X$ whose associated matching is $M$. In \cite{ayala} the authors extend these results to the case of locally finite complexes. They consider discrete Morse functions that are proper, which means that $f^{-1}([a,b])$ is finite for all $a<b\in\mathbb{R}$. 

For proper discrete Morse functions the proofs contained in \cite{forman} may be applied to show that if $f^{-1}([a,b])$ does not contain any critical cells for $a<b\in\mathbb{R}$, then the level subcomplexes $X(a), X(b)$ are homotopy equivalent. This allows one to study the topology of $X$ at different levels of $f$, but does not give a description or an explicit construction of a complex built solely of the critical cells. In \cite{ayala-morse_ineq1, ayala-morse_ineq2} the discrete Morse inequalities were given for proper discrete Morse functions on 1-dimensional and 2-dimensional locally finite complexes, which may be considered a step in this direction.

Call an infinite, directed simple path in $\mH_M(X)$ a decreasing ray, or simply a ray. Two rays are called equivalent if they coincide from some point. The paper \cite{ayala} characterizes the matchings coming from proper discrete Morse functions as those that are acyclic and do not contain some specific configurations of infinite paths in the graph $\mH_M(X)$, though this is only done in case of matchings with a finite number of critical cells and equivalence classes of decreasing rays.

In the present article we take a different approach and concentrate on the discrete Morse matchings, forgetting about discrete Morse functions. We prove that given an acyclic matching $M$ on an infinite (not necesarrily locally finite), regular CW-complex $X$ such that $\mH_M(X)$ contains finitely many equivalence classes of rays one may construct a CW-complex $X_M$ that has exactly one cell for each critical cell of $X$ and one cell for each equivalence class of decreasing rays in $X_M$. This yields a version of the discrete Morse inequalities as a corollary. Since we do not assume local finiteness or finite dimensionality of the CW-complex under investigation, our results extend those of \cite{ayala-morse_ineq1, ayala-morse_ineq2}

In fact we prove a bit more than it is stated above. Our results hold for the class of admissible posets. Moreover, we prove a homological version of the theorem for homologically admissible posets. Both of these classes were introduced in \cite{minian} and are strictly larger than the class of face posets of regular CW-complexes.

The paper is organized as follows. In Section \ref{preliminaries} we give some notation used throughout the paper and define basic combinatorial concepts used in discrete Morse theory. Section~\ref{rayless} defines rayless discrete Morse matchings and proves some of their basic properties. Next, in Section \ref{reversing} we show that given a discrete Morse matching $M$ that induces finitely many equivalence classes of decreasing rays one can create a new matching that is rayless and apart from the critical cells coming from $M$ has one new critical cell for each equivalence class of rays in $M$. In Section \ref{homological} we derive from the work of Freij \cite{freij}, which was one of the main inspirations for writing this article, a homological version of the main theorem of discrete Morse theory for Morse matchings with finitely many equivalence classes of rays. Section \ref{homotopical} gives a homotopical version of the theorem. Finally, in Section \ref{remarks} we give some concluding remarks and ideas for further research.
\end{section}

\begin{section}{Preliminaries}\label{preliminaries}
Let us fix some notation. Given a \textit{poset} (partially ordered set) $P$ and a finite chain (i.e. a totally ordered subset) $C=\{x_0,x_1,\ldots,x_n\}\subseteq P$ we say $n$ is the \textit{length} of the chain $C$. For $x\in P$ let $x\!\downarrow_P=\{y\in P:y\leq x\}$, $x\!\uparrow_P=\{y\in P:y\geq x\}$ and let $\hat{x}\!\downarrow_P=x\!\downarrow_P\smallsetminus\{x\}$, $\hat{x}\!\uparrow_P=x\!\uparrow_P\smallsetminus\{x\}$. Moreover, let $\hat{x}\!\updownarrow_P=\hat{x}\!\uparrow_P\cup\hat{x}\!\downarrow_P$. If it does not lead to confusion, we shall omit the $P$ in notation, i.e. write $x\!\downarrow$, $\hat{x}\!\uparrow$, etc. $P$ is called \textit{graded} if for every $x\in P$ all maximal chains in $x\!\downarrow$ are finite and have the same length, which is then called the \textit{degree of $x$} and denoted by $\deg(x)$. $P$ is said to have \textit{finite principal ideals} if for every $x\in P$ the set $x\!\downarrow$ is finite.

Given a \textit{digraph} (directed graph without loops and multiple edges) $D$, we call a set $M$ of arrows in $D$ a \textit{matching} on $D$ if no two arrows in $M$ are adjacent. $M$ may be also thought of as a set of disjoint pairs of adjacent vertices of $D$; vertices belonging to such a pair in $M$ are \textit{matched by $M$}.  A sequence $\{x_0, x_1, \ldots, x_n\}$ of elements of $D$ is called a \textit{simple path in $D$ of length $n$} if there is an arrow pointing from $x_i$ to $x_{i+1}$ and $x_i\not=x_j$ for all $i,j\in\mN$. $D$ is said to contain a \textit{cycle} if there is a simple path $\{x_0, x_1,\ldots, x_n\}$ in $D$ with an arrow pointing from $x_n$ to $x_0$; otherwise, $D$ is called \textit{acyclic}.

By $\mH(P)$ we denote the \textit{Hasse diagram} of the poset $P$, which is the digraph whose vertices are the elements of $P$ and which has an arrow from $x$ to $y$ for each pair $x,y\in P$ such that $x\succ y$ ($x$ covers $y$). Given a matching $M$ on $\mH(P)$ by $\mH_M(P)$ we denote the directed graph obtained from $\mH(P)$ by reversing the arrows which are in $M$. A matching on $\mH(P)$ is called \textit{acyclic} if $\mH_M(P)$ is an acyclic digraph.

We use throughout the paper the same notation for a geometric CW-complex and for the discrete set of its cells. In particular given an abstract simplicial complex $K$ we write $K$, and not $|K|$ or $\overline{K}$, for its geometric realization.

The \textit{order complex} of $P$, denoted by $\mK(P)$, is the abstract simplicial complex whose simplices are the finite chains of $P$. Given a regular CW-complex $X$ we define $\mP(X)$ to be its \textit{face poset}, which is the poset of cells of $X$ ordered by inclusion. Note that $\mK(\mP(X))$ is just the barycentric subdivision of $X$.
\end{section}

\begin{section}{Rayless Morse matchings}\label{rayless}
\begin{definition}
A sequence $\{x_0, x_1, x_2\ldots\}$ of elements of a digraph $D$ is called a \textit{(directed) ray} if for all $i\in\mN$ there is an arrow in $D$ from $x_i$ to $x_{i+1}$ and $x_i\not=x_j$ for all $i,j\in\mN$. $D$ is called a \textit{rayless digraph} if it does not contain a ray.
\end{definition}

\begin{definition}
An acyclic matching $M$ on the Hasse diagram $\mH(P)$ of a poset $P$ is called a \textit{Morse matching on $P$}. An element $x\in P$ is called \textit{critical} if it is not matched with any other element in $P$.  By a Morse matching on a regular CW-complex $X$ we understand a Morse matching on $\mP(X)$. A Morse matching $M$ is \textit{rayless} if the digraph $\mH_M(P)$ is rayless. 
\end{definition}

One should keep in mind that a Morse matching on a poset in the above sense in general does not correspond to any discrete Morse function (for example, when the poset $P$ is far from being a face poset of a CW-complex).

Note that in \cite{ayala} directed rays are called decreasing rays, which reflects the fact that any discrete Morse function associated with a discrete Morse matching (if there is one) decreases along every directed ray induced by the matching. This stays in contrast to increasing rays, which are sequences $x_0, x_1, x_2\ldots$ of distinct elements of a digraph with arrows pointing from $x_{i+1}$ to $x_i$ for all $i\in\mN$. Our definition of a rayless digraph allows it to contain an increasing ray.

\begin{definition}
Given a graded poset $P$ together with a Morse matching $M$ and a vertex $x\in P$ of degree $n$, we define the set \[M_+(x)=\begin{cases}\{y\prec z: y\not=x\} & \text{if } x \text{ is matched with some vertex } z \text{ of degree } n+1\\ \emptyset & \text{otherwise}\end{cases}.\] 

Let $D_0(x)=\{x\}$, $D_n(x)=D_{n-1}(x)\cup\bigcup_{y\in D_{n-1}(x)}M_+(y)$ for $n>0$ and $D_M(x)=\bigcup_{n\in\mN}D_n(x)$. Equip $D_M(x)$ with arrows facing from $y\in D_M(x)$ to all $z\in M_+(y)$.
\end{definition}

Recall the following result, known as the K{\"o}nig's Infinity Lemma.
\begin{lemma}\label{konig}Every infinite, directed, rooted tree $T$ with arrows directed away from the root such that every vertex of $T$ has finite degree contains a directed ray. 
\end{lemma}

\begin{lemma}\label{tree_lemma}
If $P$ is a graded poset with finite principal ideals equipped with a rayless Morse matching $M$, then for every $x\in P$ the digraph $D_M(x)$ is acyclic and finite.
\end{lemma}
\begin{proof}
Acyclicity of $D_M(x)$ clearly follows from the definition of $D_M(x)$ and $M$ being acyclic. Consider a spanning tree $T$ in $D_M(x)$. By definition (and acyclicity) of $D_M(x)$ it has edges directed away from the root $x$. Degree of each vertex in $T$ is finite  since $P$ has finite principal ideals. Because $M$ is rayless, by Lemma \ref{konig} the tree $T$ is finite and, as $T$ is a spanning tree in $D_M(x)$, the latter digraph is also finite.
\end{proof}

\begin{definition}
For $P$ a graded poset with finite principal ideals equipped with a rayless Morse matching $M$ and $x\in P$ let $L_M(x)$ denote the length of the longest directed simple path in $D_M(x)$.
\end{definition}

\end{section}

\begin{section}{Reversing rays}\label{reversing}
Let $r=\{r_1,r_2,r_3,\ldots\}, s=\{s_1,s_2,s_3,\ldots\}$ be two rays in a digraph $D$. We define (after \cite{ayala}) the rays $r, s$ to be \textit{equivalent} if there exist $M,N\in\mN$ such that $r_{M+i}=s_{N+i}$ for all $i\in\mN$. The equivalence class of $r$ under this relation is denoted by $[r]$.

We say that $r$ has a \textit{bypass starting at $r_n$} if there is a simple path in $D$ not contained in $r$ that goes from $r_n$ to $r_{n+k}$ for some $k>0$. The ray $r$ is called a \textit{multiray} if for every $N\in\mN$ the ray $r$ has a bypass starting at $r_{N+i}$ for some $i\in\mN$.

Let $M$ be a Morse matching on a graded poset $P$ and let $r=\{r_1,r_2,r_3,\ldots\}$ be a ray in $\mH_M(P)$. Let
\begin{align*}i_0& =\min\{i:\exists_{j\in\mN} \deg(r_j)=i\}\\
j_0& =\min\{j:\deg(r_j)=i_0\}.
\end{align*}
One easily notes that since $M$ is a matching the ray $r'=r_{j_0}, r_{j_0+1}, r_{j_0+2},\ldots$ consists only of elements of degree $i_0$ and $i_0+1$. We say $i_0$ is the \textit{degree} of the ray $r$, or call $r$ an \textit{$i_0$-ray}. Clearly, equivalent rays have the same degree.




\begin{lemma}\label{prop_multiray}
If an acyclic digraph $D$ contains a multiray, then it contains $2^{\aleph_0}$ equivalence classes of rays.
\end{lemma}
\begin{proof}
In order to avoid notational horrors we shall only sketch the idea of the proof. Given a multiray $r$ one may construct an infinite path $s$ in $D$ which is the concatenation of an infinite sequence of bypasses of $r$ and fragments of $r$ joining two consecutive bypasses. By acyclicity of $D$ the infinite path $s$ is simple, i.e. it is a ray.

There are $\aleph_0$ bypasses in the above sequence. Index them with the natural numbers. Now, for each subset $A$ of $\mN$ produce a ray $s_A$ in $D$ by replacing with fragments of $r$ the bypasses in the above sequence that are indexed by the elements of $A$. 

For $A,B\subseteq \mN$ the rays $s_A, s_B$ are equivalent if and only if the symmetric difference of $A$ and $B$ is finite, so the construction gives $2^{\aleph_0}$ distinct equivalence classes of rays.
\end{proof}

For a Morse matching $M$ on a graded poset $P$ denote by $\mC_M(P)$ the set of critical elements of $P$ with respect to $M$ and by $\mR_M(P)$ the family of equivalence classes of rays in $\mH_M(P)$. Let $\mC_M^n(P)=\{c\in\mC_M(P):\deg(c)=n\}$ and $\mR^n_M(P)=\{[r]\in\mR_M(P):r\text{ is an } n\text{-ray}\}$.

\begin{lemma}\label{reversing_lemma}
Let $P$ be a graded poset equipped with a Morse matching $M'$. Let $r$ be a ray in $P$ that is not a multiray. Then a Morse matching $M$ on $P$ exists such that $\mathcal{C}_{M}(P)=\mathcal{C}_{M'}(P)\cup \{c_{[r]}\},$ , where $c_{[r]}\in P\smallsetminus \mathcal{C}_{M'}(P)$, $\deg(c_{[r]})=\deg(r)$, and $\mathcal{R}_M(P)=\mathcal{R}_{M'}(P)\smallsetminus \{[r]\}$.
\end{lemma}
\begin{proof}
Let $r=\{r_0,r_1,\ldots\}$, $\deg(r)=n$. Since $r$ is not a multiray, me may assume $r$ has no bypasses. Moreover, we assume $\deg(r_0)=n$. Thus, for all $k\in\mN$ the pairs $(r_{2k}, r_{2k+1})$ belong to the matching $M'$. Let $M=M'\smallsetminus \{(r_{2k}, r_{2k+1}):k\in\mN\} \cup \{(r_{2k+1}, r_{2k+2}):k\in\mN\}$. 

We will show that $M$ is a Morse matching on $P$ such that $\mathcal{C}_{M}(P)=\mathcal{C}_{M'}(P)\cup\{r_0\}$ and every equivalence class in the family $\mathcal{R}_{M}(P)$ contains a representative belonging to a class in $\mathcal{R}_{M'}(P)\smallsetminus\{[r]\}$.

Clearly, $M$ is a matching. Suppose that $\mH_{M}(P)$ contains a cycle $c=\{c_0,c_1,\ldots,c_k\}$. Then some elements consecutive in the cycle $c$ belong to $r$; otherwise, $c$ would be a cycle in $\mH_{M'}(P)$. We may assume $c_0,c_k\in r$. Let $c_{i},c_{i+1}\ldots,c_{i+m}\not\in r$ be such that $c_{i-1},c_{i+m+1}\in r$. Thus $c_{i-1}=r_{a}, c_{i+m+1}=r_{b}$ for some $a,b\in\mN$. If $a< b$, then $\{c_{i-1}, c_i,\ldots, c_{i+m+1}\}$ would be a cycle in $\mH_{M'}(R)$, which is impossible. This implies $a<b$, but it means $\{c_{i-1}, c_i,\ldots, c_{i+m+1}\}$ is a bypass of $r$ in $\mH_{M'}(R)$, which is a contradiction. Therefore, $\mH_{M}(P)$ is acyclic.

It is obvious that $\mathcal{C}_{M}(P)=\mathcal{C}_{M'}(P)\cup\{r_0\}$. We now need to prove that essentialy no new rays are created when switching from $M'$ to $M$. But, using arguments similar as above, one shows that the intersection of $r$ with any ray $s$ in $\mH_{M}(P)$ must be finite. Thus, $s$ is equivalent to a ray in $\mH_{M'}(P)$.
\end{proof}

\begin{theorem}\label{reversing_thm}
Let $P$ be a graded poset equipped with a Morse matching $M'$ such that $\mathcal{R}_{M'}(P)$ is finite. Then a rayless Morse matching $M$ on $P$ exists such that \[\mathcal{C}_{M}(P)=\mathcal{C}_{M'}(P)\cup \{c_{[r]}:[r]\in\mathcal{R}_{M'}(P)\},\] where $c_{[r]}\in P\smallsetminus \mathcal{C}_{M'}(P)$, $c_{[r]}\not=c_{[r']}$ for $[r]\not=[r']$ and $\deg(c_{[r]})$ equals the degree of the ray $r$. 
\end{theorem}
\begin{proof}
By Lemma \ref{prop_multiray} the digraph $\mH_{M'}(P)$ does not contain a multiray. The theorem follows by induction from Lemma \ref{reversing_lemma}.
\end{proof}

One could be tempted to prove, using transfinite induction, an analogue of Theorem \ref{reversing_thm} in the situation of infinitely many equivalence classes of rays, but no multirays. However, a problem arises at limit ordinals: in some cases after reversing infinitely many rays a new ray may be created. This problem does not occur in the case of discrete Morse matchings on $1$-dimensional regular CW-complexes, as the following theorem shows.

\begin{theorem}\label{reversing_thm2}
Let $X$ be a $1$-dimensional regular CW-complex and let $M'$ be a Morse matching on $X$. A rayless Morse matching $M$ on $X$ exists such that \[C_M(\mP(X))=C_{M'}(\mP(X))\cup\{c_{[r]}:[r]\in\mR_{M'}(\mP(X))\},\] where $c_{[r]}\in\mP(X)\smallsetminus C_{M'}(\mP(X))$, $c_{[r]}\not=c_{[r']}$ for $[r]\not=[r']$ and $\deg(c_{[r]})=0$ for all $[r]$.
\end{theorem}
\begin{proof}
We only sketch the proof, as we do not find the theorem very important. First note that, since cells of degree $1$ contain exactly two $0$-cells, for a Morse matching $N$ on $X$ no ray $r$ in $\mH_{N}(X)$ such that $\deg(r_0)=\deg(r)$ has a bypass. For the same reason there is no directed path between elements of two rays $r,r'$ in $\mH_{N}(X)$ such that $[r]\not=[r']$. The first observation allows us to use Lemma \ref{reversing_lemma}. The second assures us that if we use Lemma \ref{reversing_lemma} in a transfinite induction process, at the limit ordinals no new rays are created.
\end{proof}

It should be noted that the ray-reversal technique utilised in the above proofs was earlier used in \cite{ayala-number}.
\end{section}

\begin{section}{Homological discrete Morse theory}\label{homological}
In \cite{freij} Freij extracts the algebraic essence of discrete Morse theory and observes that given a chain complex with an operator resembling the Morse matching one can define a discrete Morse chain complex whose homology groups are isomorphic to those of the given complex, provided that the mentioned operator satisfies a condition close to nilpotency, which is an algebraic analogue of the lack of directed rays. 

In this section we derive, as a direct corollary of Freij's results and Theorems \ref{reversing_thm}, \ref{reversing_thm2}, a homological version of the main theorem of discrete Morse theory, which for finite complexes was developed in \cite[Sections 6 - 8]{forman}, and of the discrete Morse inequalities \cite[Corollaries 3.6 and 3.7]{forman} for Morse matchings on homologically admissible posets inducing finitely many equivalence classes of rays and for Morse matchings on $1$-dimensional regular CW-complexes.

The following definitions were introduced in the finite case by Minian \cite{minian}.

\begin{definition}
A poset $P$ is called:
\begin{itemize}
\item \textit{cellular}, if $P$ is graded with finite principal ideals and $\mK(\hat{x}\!\downarrow)$ has the homology of a \mbox{$(\deg(x)-1)$-dimensional} sphere for all $x\in X$;
\item \textit{homologically admissible} (or h-admissible), if $P$ is graded with finite principal ideals, \mbox{$\mK(\hat{x}\!\downarrow\smallsetminus\{y\})$} is connected and the homology groups $H_n(\mK(\hat{x}\!\downarrow\smallsetminus\{y\}))$, $n\geq 1$  are trivial for all $x\in X$ and every maximal element $y\in\hat{x}\!\downarrow$.
\end{itemize}

We shall call elements of degree $n$ of cellular posets \textit{$n$-cells}.
\end{definition}

Note that if $X$ is a regular CW-complex, then $\mP(X)$ is a cellular poset and an $n$-cell in $\mP(X)$ is just an $n$-cell in $X$.

As in \cite[Section 3]{minian} one can prove that the homology (with integer coefficients) of an infinite cellular poset $P$ can be computed using the \textit{cellular chain complex} $(\mC_*,d)$, which in dimension $n$ consists of the free abelian group generated by the $n$-dimensional elements of $P$. Moreover, every h-admissible poset is cellular and the differential $d\colon\mC_n\to \mC_{n-1}$ of its cellular chain complex has the form $d(x)=\sum_{y\prec x}\epsilon(x,y)y$ with $\epsilon(x,y)\in\{-1,1\}$.

Given a discrete Morse matching $M'$ on a h-admissible poset $P$ such that $\mR_{M'}(P)$ is finite we replace it as in Theorem \ref{reversing_thm} by a rayless Morse matching $M$. Now we may define a homomorphism \mbox{$V\colon\mC_{*}\to\mC_{*+1}$}. Given an $n$-cell $y\in P$ we put $V(y)=-\epsilon(x,y)x$, where $x\in P$ is the unique $(n+1)$-cell matched by $M$ with $y$ if such a cell exists, and otherwise we put $V(y)=0$. $V$ extends linearly to whole $\mC_{*}$.

\begin{proposition}
$V$ is a gradient vector field in the sense of \cite[Definition 3.1]{freij}, i.e. the following conditions are satisfied:
\begin{itemize}
\item $V\colon\mC_n\to \mC_{n+1}$ is a homomorphism for each $n\in\mN$,
\item $V^2=0$,
\item for every $x\in C$ there is some $k\in\mN$ with $V(1-dV)^k x=0$.
\end{itemize}
\end{proposition}
\begin{proof}
The first condition is obviously satisfied. The second follows from the fact that $M$ is a matching. To prove the third condition first assume that $x\in C$ is an $n$-cell. Note that if $M_+(x)\not=\emptyset$, then $(1-dV)x=\sum_{y\in M_+(x)}\epsilon_y y$, where $\epsilon_y\in\{-1,1\}$, and otherwise $(1-dV)x=x$. By Lemma \ref{tree_lemma} we must thus have $V(1-dV)^k x=0$ for $k$ greater than $L_M(x)$. This proves the condition is true for $x$ being a cell. The general case follows from the linearity of $V$.
\end{proof}

This allows us to define the Morse complex $(\mM_*,\overline{d})$.

\begin{corollary}\label{homological_main_thm}
There exists a chain complex $(\mM_*,\overline{d})$, where $\mM_n$ is a free abelian group spanned by $\mC_{M'}^n(P)\cup\mR_{M'}^n(P)$ (or equivalently by $\mC_{M}^n(P)$), such that the homology groups of $(\mM_*,\overline{d})$ and $(\mC_*,d)$ are isomorphic.
\end{corollary}
\begin{proof}
This follows form \cite[Chapter 3]{freij}.
\end{proof}

As a further corollary, derived in the standard way, as in \cite[pages 28-31]{milnor}, the strong and weak Morse inequalities are true in this setting.

\begin{corollary}
Let $m_n$ denote the number of critical $n$-cells in $P$, $r_n$ denote the number of equivalence classes of $n$-rays in $P$ and $b_n$ be the $n$-th Betti number of $P$. For all $N\in\mN$:
\begin{enumerate}
\item $\sum_{i\leq N} (-1)^{N-i} (m_i+r_i)\geq \sum_{i\leq N} (-1)^{N-i} b_i$, provided that all $m_i$ are finite for $i\leq N$,
\item $b_N\leq m_N+r_N$,
\item $\sum_{i=0}^{\infty} (-1)^i b_i=\sum_{i=0}^{\infty} (-1)^i (m_i+r_i)$, provided that all $m_i$ are finite and all but finitely many $m_i$ are zero.
\end{enumerate}
\end{corollary}

Using Theorem \ref{reversing_thm2} instead of Theorem \ref{reversing_thm} we could repeat the above argument for arbitrary Morse matchings on $1$-dimensional regular CW-complexes, obtaining in particular the discrete Morse inequalities of \cite[Theorem 3.1]{ayala-morse_ineq2} as a corollary.
\end{section}

\begin{section}{Homotopical discrete Morse theory}\label{homotopical}
While the results of Freij \cite{freij} utilised in Section \ref{homological} allowed us to derive at once the homological version of the main theorem of discrete Morse theory in the infinite setting, up to the best of the author's knowledge no similar results concerning the topological (or, more precisely, homotopical) version of the theorem were published. In this section we shall fill the gap.

As in the previous section, following \cite{minian} we will work with a class of posets strictly larger than the class of face posets of regular CW-complexes. For the latter class some technicalities, such as Lemma \ref{multiple_glue}, could be omitted.

\begin{definition}
A poset $P$ is called \textit{admissible} if $P$ is graded with finite principal ideals and $\mK(\hat{x}\!\downarrow\smallsetminus\{y\})$ is contractible for all $x\in X$ and all $y$ maximal in $\hat{x}\!\downarrow$.
\end{definition}

Of course, admissible posets are homologically admissible. Also, face posets of regular CW-complexes are admissible.

\begin{lemma}\label{multiple_glue}
Let $P$ be an admissible poset and let $A=\{x_i:i\in I\}$ be a set of elements maximal in $P$. Then $\mK(P)$ is homotopy equivalent to the CW-complex $X=\mK(P\smallsetminus A)\cup \bigcup_{i\in I} e_i$ obtained by attaching to $\mK(P\smallsetminus A)$ one $\deg(x_i)$-dimensional cell $e_i$ for each $i\in I$.
\end{lemma}
\begin{proof}
Let $p_i=\deg(x_i)$. Since $P$ is admissible, by \cite[Proposition 2.10]{minian} there exists, for each $i\in I$, a homotopy equivalence $f_i\colon S^{p_i-1}\to \mK(\hat{x_i}\!\downarrow)$. Let the attaching map of $e_i$ be $j_i\circ f_i$, where \mbox{$j_i\colon \mK(\hat{x_i}\!\downarrow)\hookrightarrow \mK(P\smallsetminus A)$} and let $\phi_i$ denote the characteristic map of $e_i$.

We have the following commutative diagram of spaces
\[\xymatrix@M=5pt{
& \bigsqcup_{i\in I}\mK(\hat{x_i}\!\downarrow)\ar@{_(->}'[d][dd]\ar[rr]^{\bigsqcup_{i\in I}j_i} & & \mK(P\smallsetminus A)\ar@{_(->}[dd]\\
\bigsqcup_{i\in I}S_i\ar@{_(->}[dd]\ar[ur]^{\bigsqcup_{i\in I}f_i}\ar[rr]^(0.6){\bigsqcup_{i\in I}j_i\circ f_i} & & \mK(P\smallsetminus A)\ar@{_(->}[dd]\ar@{=}[ur]\\ 
& \bigsqcup_{i\in I}\mK(x_i\!\downarrow)\ar@{^(->}'[r][rr] & & \mK(P)\\
\bigsqcup_{i\in I}D_i\ar[rr]^{\bigsqcup_{i\in I}\phi_i}\ar[ur]^{\bigsqcup_{i\in I}F_i} & & \mK(P\smallsetminus A)\cup \bigcup_{i\in I} e_i\ar[ur]^F
}\]
where $S_i=S^{p_i-1}$, $D_i=D^{p_i}$, the front and the back squares are pushouts, $F_i([x,t])=[f(x),t]$ for $x\in D_i$ (we treat $D_i, \mK(x\!\downarrow)$ as cones over $S_i, \mK(\hat{x}\!\downarrow)$ respectively) and the map $F$ between pushouts is determined by the maps $\bigsqcup_{i\in I}f_i, \bigsqcup_{i\in I}F_i, \id_{\mK(P\smallsetminus A)}$.

By the gluing theorem for adjunction spaces proved in \cite[Theorem 7.5.7]{brown} the map $F$ is a homotopy equivalence.
\end{proof}

We may now state and prove the main theorem of the present article.

\begin{theorem}\label{mainthm}
Let $P$ be an admissible poset equipped with a discrete Morse matching $M$ such that $\mR_M(P)$ is finite. Then $\mK(P)$ is homotopy equivalent to a CW-complex $P_M$ that for all $n\in\mN$ has exactly one $n$-cell for each element of $\mC^n_M(P)\cup \mR^n_M(P)$. 
\end{theorem}
\begin{proof}
Theorem \ref{reversing_thm} allows us to assume the matching $M$ is rayless.

We define $P^0$ to be the set of critical $0$-cells of $P$. Given $P^{n-1}$ we define $P^n_*$ as the union of $P^{n-1}$, the set of noncritical $(n-1)$-cells of $P$ not contained in $P^{n-1}$ and their $n$-dimensional matches. Now, $P^n$ is defined to be the union of $P^n_*$ and the set of critical $n$-cells of $P$. 
The Morse matching on $P$ restricts to Morse matchings on $P^n, P^{n+1}_*$ for all $n\in\mN$. 

Let $i\in\mN$. We will show that $\mK(P^{i+1}_*)$ strong deformation retracts to $\mK(P^i)$. In order to do this, consider the sets $L_n=\{x\in P^{i+1}_*:\deg(x)=i \textrm{ and } L_M(x)= n\}$ for $n\in\mN$. Note that for an $i$-dimensional cell $x\in P^{i+1}_*$ we have $x\in L_0$ if and only if $x\in P^i$. For $x\in L_n$, $n>0$, let $m_x$ be the unique $(i+1)$-cell matched with $x$. For $n>0$ we define $\overline{L_n}=L_n\cup \{m_x:x\in L_n\}$. Let $P^i_k=P^i\cup \bigcup_{1\leq n\leq k}\overline{L_n}$. Because $M$ is rayless for each $i$-cell $x\in P^i_*$ the number $L_M(x)$ is defined and thus $\bigcup_{k\in\mathbb{N}}P^i_k=P^{i+1}_*$. 

Let $x\in L_n, n>0$. We will show that \begin{equation}\label{eq1}x\!\uparrow_{P^i_n}=\{m_x\}.\end{equation} Suppose $x$ is contained in some $y\in P^i_n$ with $y\not=m_x$. Therefore, $y\in \overline{L_k}$ for some $k\leq n$, which means $y=m_z$ for some $z\in L_k$. But this means $x\in M_+(z)$, so $L_M(z)\geq L_M(x)+1 = n+1$, so $z\not\in L_k$.

We will now prove that for any $(i+1)$-cell $y\in P^i_n$ we have \begin{equation}\label{eq2}y\!\downarrow_P\subseteq P^i_n.\end{equation} In fact, we only need to show that all elements of $P$ covered by $y$ belong to $P^i_n$, since cells of dimension lower than $i$ are all contained in $P^i\subseteq P^i_n$. As $y\in P^i_n$, there is some $i$-cell $x$ with $m_x=y$ and $L_M(x)\leq n$. If for any $z$ covered by $y$ we had $L_M(z)=n'>n$, then we would have $z\!\uparrow_{L_{n'}}\supseteq \{m_x,m_z\}$, which, because of (\ref{eq1}), is impossible, since $m_z\not=m_x$.

By (\ref{eq1}) for every $x\in L_n$ the set $\hat{x}\!\updownarrow_{P^i_n}$ has a greatest element $m_x$, and thus $\mK(\hat{x}\!\updownarrow_{P^i_n})$ is contractible as a cone and $\mK(P^i_n)$ strong deformation retracts to $\mK(P^i_n\smallsetminus\{x\})$ by \cite[Lemma 2.9]{minian}.

Now, $\mK(\widehat{m_x}\!\updownarrow_{P^i_n\smallsetminus\{x\}})$ is contractible, which follows from (\ref{eq2}) and $P$ being admissible. Again by \cite[Lemma 2.9]{minian}, $\mK(P^i_n\smallsetminus\{x\})$ strong deformation retracts to $\mK(P^i_n\smallsetminus\{x,m_x\})$. Composition of these two retractions gives a strong deformation retraction $r_x\colon \mK(P^i_n)\to \mK(P^i_n\smallsetminus\{x,m_x\})$. Note that $r_x$ may be chosen so that $r_x(\mK(m_x\!\downarrow))\subseteq \mK(\hat{m_x}\!\downarrow\smallsetminus\{x\})$. This allows to combine the deformation retractions $r_x$ for all $x\in L_n$ into a strong deformation retraction $r^i_n\colon \mK(P^i_n)\to \mK(P^i_{n-1})$. Let us define the composition $R^i_n=r^i_1\circ r^i_2\circ\dots\circ r^i_n\colon \mK(P^i_{n})\to \mK(P^i)$. Clearly, $R^i_n$ are strong deformation retractions and we have a containment of graphs of functions $R^i_n\subseteq R^i_{n+1}$. The map \mbox{$R^i=\bigcup_{n\in\mN}R^i_n\colon \mK(P^{i+1}_*)\to \mK(P^{i})$} is continuous, since any finite subcomplex of $\mK(P^{i+1}_*)$ is contained in some $\mK(P^i_n)$, so $R^i$ restricted to any finite subcomplex is just some $R^i_n$, which is continuous.

We will show that $R^i\colon \mK(P^{i+1}_*)\to \mK(P^i)$ is a weak homotopy equivalence. Surjectivity of the induced map $\pi_k(R^i)$ follows from $R^i$ being a retraction. To prove injectivity of $\pi_k(R^i)$ consider a $[p]\in\pi_k(\mK(P^{i+1}_*))$. Since the spheres are compact spaces, $p$ maps $S^k$ onto a compact subcomplex of $\mK(P^{i+1}_*)$, which is contained in some $\mK(P_{i}^n)$, and since $R^i_n$ is a strong deformation retraction, the map $p$ may be homotoped to the map $R^i_n\circ p=R\circ p$. Therefore, if $R\circ p$ is nullhomotopic, $p$ is nullhomotopic, which proves the injectivity of $\pi_k(R^i)$. By the Whitehead Theorem \cite[Theorem 4.5]{hatcher} this means $R^i$ is a strong deformation retraction.

We will now define inductively, for $i\in\mN$, CW-complexes $M^i$ and homotopy equivalences $h^i\colon \mK(P^i)\to M^i$ such that $M^i\subseteq M^{i+1}$ and $h^{i+1}$ restricted to $\mK(P^{i})$ equals $h^{i}$. Let $M^0=\mK(P^0)$, $h^0=\id_{P^0}$. Given $M^i$ and $h^i$ we shall define $M^{i+1}, h^{i+1}$. Let \[A=P^{i+1}\smallsetminus P^{i+1}_*=\{x\in P: \deg(x)=i, x \textrm{ critical}\}.\] By Lemma \ref{multiple_glue} the simplicial complex $\mK(P^{i+1})$ is homotopy equivalent to a CW-complex $\mK(P^{i+1}_*)\cup \bigcup_{x\in A}e_x$.  Denote the attaching map of the $(i+1)$-cell $e_x$ by $f_x$. Since $R^i\colon \mK(P^{i+1}_*)\to \mK(P^{i})$ is a strong deformation retraction, it induces a homotopy equivalence rel $\mK(P^i)$, denoted by $\widetilde{R^{i+1}}$, between $\mK(P^{i+1}_*)\cup \bigcup_{x\in A}e_x$ and the CW-complex $\mK(P^i)\cup \bigcup_{x\in A}\widetilde{e_x}$ obtained by attaching to $\mK(P^{i})$ the set of critical $(i+1)$-cells $\{\widetilde{e_x}\}_{x\in A}$ via attaching maps $R^i\circ f_x$. Now, the homotopy equivalence $h^i\colon \mK(P^i)\simeq M^i$ induces a homotopy equivalence $\widetilde{h^{i+1}}$ between $\mK(P^i)\cup \bigcup_{x\in A}\widetilde{e_x}$ and the CW-complex $M^{i+1}$, which is obtained by attaching to $M^{i}$ the set of critical $(i+1)$-cells $\{\widehat{e_x}\}_{x\in A}$ via attaching maps $h^i\circ R^i\circ f_x$. Let $h^{i+1}=\widetilde{h^{i+1}}\circ \widetilde{R^{i+1}}$.

From the construction it is clear that $P_M=\bigcup_{i\in\mN}M^i$ has exactly one $n$-cell for each critical cell of the matching $M$ on $P$ of degree $n$. The map $h_M=\bigcup_{i\in\mN}h^i\colon\mK(P)\to P_M$ is continuous, since it is continuous on every compact subcomplex.

We will show $h_M$ is a weak homotopy equivalence. First, let $[p]\in\pi_k(\mK(P))$. Then, by compactness, the image of $p$ is contained in some $\mK(P^i)$. Let $j=\max(i,k+1)$. If $h_M\circ p=h^j \circ p$ is nullhomotopic in $P_M$, then by the Cellular Approximation Theorem \cite[Theorem 4.8]{hatcher} it is nullhomotopic in $M^{k+1}$, and since $h^j$ is a homotopy equivalence, $p$ is nullhomotopic in $\mK(P^j)\subseteq \mK(P)$. Therefore, $\pi_k(h_M)$ is injective. 

Now, let $[q]\in\pi_k(P_M)$. Using again the Cellular Approximation Theorem we may assume that the image of $q$ is contained in $M^{k}$. Since $h^k$ is a homotopy equivalence, there is some $p\colon S^k\to \mK(P^k)\subseteq \mK(P)$ with $q\simeq h^k\circ p = h\circ p$. This proves the surjectivity of $\pi_k(h_M)$. By the Whitehead Theorem \cite[Theorem 4.5]{hatcher}, $h_M$ is a homotopy equivalence.
\end{proof}

Note that a version of Theorem \ref{mainthm} for arbitrary Morse matchings on $1$-dimensional regular CW-complexes is also true and can be easily proved using Theorem \ref{reversing_thm2}.
\end{section}

\begin{section}{Final remarks and open problems}\label{remarks}
The notion of discrete Morse function was absent throughout Sections \ref{preliminaries} - \ref{homotopical} because the author finds Morse matchings more flexible than discrete Morse functions. However, given a rayless Morse matching $M$ on a regular CW-complex $X$ (or more generally, on a good poset) we may construct a discrete Morse function $f$ on $X$ that is self-indexing (i.e. its value on each critical cell equals the dimension of that cell) whose induced Morse matching is $M$. Let $X^i$, $X^i_*$ and the corresponding $L_n$ for $i,n\in\mN$ be defined as in the proof of Theorem \ref{mainthm}. Fix $i\in\mN$. For $x\in L_n$ put $f(x)=n+(1-\frac{1}{2^{2n}})$. If $n>0$, put $f(m_x)=n+(1-\frac{1}{2^{2n-1}})$. Performing this construction for all $i\in\mN$ and all $n\in\mN$ we get, as the reader will easily check, a self-indexing discrete Morse function on $X$ whose corresponding matching is $M$.

The function obtained in the previous paragraph is, however, in general not proper, since for example we do not assume that $M$ has a finite number of critical cells. For infinite CW-complexes it is often the case that a discrete Morse function cannot be both proper and self-indexing. In fact, for complexes of large cardinality a discrete Morse function cannot even be taken injective, as it is required in some of Forman's proofs. This leads one to the idea of considering discrete Morse functions that are not real-valued, but their codomain lies in a bigger ordered set, e.g. in a lexicographical product of some sufficiently large ordinal and the real line. This could make it possible to mimick the proofs of \cite{forman}.

Another direction of research is understanding and handling the multirays and the situation of infinitely many rays in complexes of dimension greater than $1$. At least in the locally finite case using the methods of locally finite homology, homology at infinity and proper homotopy theory seems to be a promising approach, that in the smooth case has led to some new results \cite{gemmeren}. 
\end{section}

\vspace{2mm}
\noindent\textbf{Acknowledgement:} During the prepartation of this paper the author has been supported by the joint PhD programme `\'Srodowiskowe Studia Doktoranckie z Nauk Matematycznych'.

%
%

\end{document}